\documentclass[reqno]{amsart}
\usepackage{hyperref}

\usepackage{amsmath,amssymb,amsthm}
\usepackage{enumerate}
\usepackage{color}
\usepackage{epsfig}
\usepackage{graphics}
\usepackage{graphicx}
\usepackage{subfigure}

\begin{document}
\title[\hfilneg  \hfil ]{an invertibility criterion in a c*-algebra acting on the hardy space with applications to composition operators}

\author[u. g\"{u}l, b. koca]{u\u{g}ur g\"{u}l and beyaz ba\c{s}ak koca}

\address{u\u{g}ur g\"{u}l,  \newline
Hacettepe University, Department of Mathematics, 06800, Beytepe,
Ankara, TURKEY}
\email{\href{mailto:gulugur@gmail.com}{gulugur@gmail.com}}
\address{beyaz ba\c{s}ak koca,  \newline
Istanbul University, Department of Mathematics, 34134, Vezneciler,
Istanbul, TURKEY}
\email{\href{mailto:basakoca@istanbul.edu.tr}{basakoca@istanbul.edu.tr}}


\thanks{Submitted May 2, 2017}


\begin{abstract}
    In this paper we prove an invertibility criterion for certain
    operators which is given as a linear algebraic combination of
    Toeplitz operators and Fourier multipliers acting on the Hardy space of the unit disc.
    Very similar to the case of Toeplitz operators we prove that such operators
    are invertible if and only if they are Fredholm and their
    Fredholm index is zero. As an application we prove that for
    ``quasi-parabolic" composition operators the spectra and the
    essential spectra are equal.
\end{abstract}

\maketitle
\newtheorem{theorem}{Theorem}
\newtheorem{acknowledgement}[theorem]{Acknowledgement}
\newtheorem{algorithm}[theorem]{Algorithm}
\newtheorem{axiom}[theorem]{Axiom}
\newtheorem{case}[theorem]{Case}
\newtheorem{claim}[theorem]{Claim}
\newtheorem{conclusion}[theorem]{Conclusion}
\newtheorem{condition}[theorem]{Condition}
\newtheorem{conjecture}[theorem]{Conjecture}
\newtheorem{corollary}[theorem]{Corollary}
\newtheorem{criterion}[theorem]{Criterion}
\newtheorem{definition}[theorem]{Definition}
\newtheorem{example}[theorem]{Example}
\newtheorem{exercise}[theorem]{Exercise}
\newtheorem{lemma}[theorem]{Lemma}
\newtheorem{notation}[theorem]{Notation}
\newtheorem{problem}[theorem]{Problem}
\newtheorem{proposition}[theorem]{Proposition}
\newtheorem{remark}[theorem]{Remark}
\newtheorem{solution}[theorem]{Solution}
\newtheorem{summary}[theorem]{Summary}
\newtheorem*{thma}{Theorem A}
\newtheorem*{thmb}{Theorem B}
\newtheorem*{thmc}{Theorem C}
\newtheorem*{thmd}{Theorem D}
\newcommand{\norm}[1]{\left\Vert#1\right\Vert}
\newcommand{\abs}[1]{\left\vert#1\right\vert}
\newcommand{\set}[1]{\left\{#1\right\}}
\newcommand{\Real}{\mathbb R}
\newcommand{\eps}{\varepsilon}
\newcommand{\To}{\longrightarrow}
\newcommand{\BX}{\mathbf{B}(X)}
\newcommand{\A}{\mathcal{A}}

\section{introduction}
In this paper we investigate the invertibility of elements in the
C*-algebra $\Psi$ generated by Toeplitz operators and Fourier
multipliers. The C*-algebra $\Psi$ is defined to be
$$\Psi=\Psi(QC,C([0,\infty]))=C^{\ast}(\{T_{\varphi}:\varphi\in QC\}\cup\{D_{\vartheta}:\vartheta\in C([0,\infty])\})$$
the C*-algebra generated by Toeplitz operators with $QC$ symbols
and Fourier multipliers with continuous symbols. This C*-algebra
was introduced by the first author in \cite{Gul1} in order to
study the spectral properties of a class of composition operators.
In \cite{Gul1}, the first author showed that $\Psi/K(H^{2})$ is a
commutative C*-algebra with identity and determined its maximal
ideal space. The maximal ideal space $\mathbb{M}$ of
$\Psi/K(H^{2})$ is found to be homeomorphic to a certain subset of
$M(QC)\times [0,\infty]$ which can be described as
$$\mathbb{M}\cong (M_{\infty}(QC(\mathbb{R}))\times[0,\infty])\cup(M(QC(\mathbb{R}))\times\{\infty\})$$
where $M_{\infty}(QC)$ is the fiber of $M(QC)$ at infinity.

 In this paper we show that if $T=\sum
 T_{\varphi_{j}}D_{\vartheta_{j}}+\sum
D_{\nu_{j}}T_{\psi_{j}}\in\Psi$ is written as a finite
 sum or as an infinite sum converging in the operator norm where
 $\psi_{j}$, $\varphi_{j}\in QC$ and $\nu_{j}$, $\vartheta_{j}\in C([0,\infty])$ then
$T$ is invertible if and only if $T$ is Fredholm and has Fredholm
index zero. We do this through constructing a homotopy
$H:[0,1]\rightarrow\Psi$ which is defined as
$$H(w):=\sum T_{\varphi_{j}}D_{\vartheta_{j}^{w}}+\sum
D_{\nu_{j}^{w}}T_{\psi_{j}}$$ where $\nu_{j}^{w}(t):=\nu_{j}(t-\ln
w)$, $\vartheta_{j}^{w}(t):=\vartheta_{j}(t-\ln w)$ and
$H(0):=T_{\varphi}$ where
$\varphi:=\sum\lambda_{j}\varphi_{j}+\sum\mu_{j}\psi_{j}$ with
$\lambda_{j}:=\lim_{t\rightarrow\infty}\vartheta_{j}(t)$ and
$\mu_{j}=\lim_{t\rightarrow\infty}\nu_{j}(t)$. We observe that
this homotopy acts continuously on finite sums hence keeps on to
act continuously on infinite sums which converge in operator norm.
 We apply this result to show that the class of composition operators that the
 first author studied in \cite{Gul1} have spectra equal to their
 essential spectra. The class of composition operators that was
 studied in \cite{Gul1} is the class of composition operators with
 symbols
 $\varphi$ which have upper half-plane re-incarnation
\begin{equation*}
\mathfrak{C}^{-1}\circ\varphi\circ\mathfrak{C}(z)=z+\psi(z)
\end{equation*}
 for a bounded analytic function
  $\psi$ satisfying $\Im(\psi(z)) > \epsilon > 0$ for all $z\in$
  $\mathbb{H}$. We call this class of composition operators
  ``quasi-parabolic".

 \section{preliminaries}

In this section we fix the notation that we will use throughout
and recall some preliminary facts that will be used in the sequel.

Let $S$ be a compact Hausdorff topological space. The space of all
complex valued continuous functions on $S$ will be denoted by
$C(S)$. For any $f\in C(S)$, $\parallel f\parallel_{\infty}$ will
denote the sup-norm of $f$, i.e. $$\parallel
f\parallel_{\infty}=\sup\{\mid f(s)\mid:s\in S\}.$$ For a Banach
space $X$, $K(X)$ will denote the space of all compact operators
on $X$ and $\mathcal{B}(X)$  will denote the space of all bounded
linear operators on $X$. The open unit disc will be denoted by
$\mathbb{D}$, the open upper half-plane will be denoted by
$\mathbb{H}$, the real line will be denoted by $\mathbb{R}$ and
the complex plane will be denoted by $\mathbb{C}$. The one point
compactification of $\mathbb{R}$ will be denoted by
$\dot{\mathbb{R}}$ which is homeomorphic to $\mathbb{T}$. For any
$z\in$ $\mathbb{C}$, $\Re(z)$ will denote the real part, and
$\Im(z)$ will denote the imaginary part of $z$, respectively. For
any subset $S\subset$ $B(H)$, where $H$ is a Hilbert space, the
C*-algebra generated by $S$ will be denoted by $C^{*}(S)$. The
Cayley transform $\mathfrak{C}$ will be defined by
\begin{equation*}
\mathfrak{C}(z)=\frac{z-i}{z+i}.
\end{equation*}
 For any $a\in$ $L^{\infty}(\mathbb{R})$ (or $a\in$
$L^{\infty}(\mathbb{T})$), $M_{a}$ will be the multiplication
operator on $L^{2}(\mathbb{R})$ (or $L^{2}(\mathbb{T})$) defined
as
\begin{equation*}
M_{a}(f)(x)=a(x)f(x).
\end{equation*}

For convenience, we remind the reader of the rudiments of the
theory of Toeplitz operators and commutative C*-algebras.

Let $A$ be a commutative Banach algebra. Then its maximal ideal
space $M(A)$ is defined as
\begin{equation*}
    M(A)=\{x\in A^{*}:x(ab)=x(a)x(b)\quad\forall a,b\in A\}
\end{equation*}
where $A^{*}$ is the dual space of $A$. If $A$ has identity then
$M(A)$ is a compact Hausdorff topological space with the weak*
topology. The Gelfand transform $\Gamma:A\rightarrow C(M(A))$ is
defined as
\begin{equation*}
    \Gamma(a)(x)=x(a).
\end{equation*}
 If $A$ is a commutative C*-algebra with
identity, then $\Gamma$ is an isometric *-isomorphism between $A$
and $C(M(A))$. If $A$ is a C*-algebra and $I$ is a two-sided
closed ideal of $A$, then the quotient algebra $A/I$ is also a
C*-algebra (see \cite{Rudin} and \cite{Murphy}).

For a Banach algebra $A$, we denote by $com(A)$ the closed ideal
in $A$ generated by the commutators
$\{a_{1}a_{2}-a_{2}a_{1}:a_{1},a_{2}\in A\}$. It is an algebraic
fact that the quotient algebra $A/com(A)$ is a commutative Banach
algebra. For $a\in A$ the spectrum $\sigma_{A}(a)$ of $a$ on $A$
is defined as
\begin{equation*}
    \sigma_{A}(a)=\{\lambda\in\mathbb{C}:\lambda e-a\ \ \textrm{is not invertible in}\ A\},
\end{equation*}
where $e$ is the identity of $A$. In particular the spectrum
$\sigma(T)$ of a linear bounded operator $T:X\rightarrow X$ where
$X$ is a Banach space is defined as
$\sigma(T):=\sigma_{\mathcal{B}(X)}(T)$. Recall that a bounded
linear operator $T$ on a Hilbert space $H$ is called Fredholm if
the range of $T$ is closed, $\dim\ker(T)$ and $\dim\ker(T^{\ast})$
are finite. The Fredholm index $ind$ is defined as
$$ind(T)=\dim(\ker(T))-\dim(\ker(T^{\ast}))$$
It is a very well known fact that (\cite{Murphy}) when the set of
Fredholm operators $\mathrm{F}\subset B(H)$ is equipped with
operator norm topology and $\mathbb{Z}$ is equipped with discrete
topology, the index function $ind:\mathrm{F}\rightarrow\mathbb{Z}$
is continuous. The essential spectrum $\sigma_{e}(T)$ of an
operator $T$ acting on a Banach
  space $X$ is the spectrum of the coset of $T$ in the Calkin algebra
  $\mathcal{B}(X)/K(X)$, the algebra of bounded linear operators modulo
  compact operators.
The following Atkinson's characterization for Fredholm operators
is also well known:
\begin{theorem}{\cite[p.28, Theorem 1.4.16]{Murphy}}\label{Atkinson}
 A bounded linear operator $T$ on a Hilbert space $H$ is Fredholm if and only if
 $T+ K(H)$ is invertible in the quotient algebra $\mathcal{B}(H)/K(H)$, where $K(H)$ is the algebra of all compact operators on $H$.
\end{theorem}

For $1\leq p < \infty$ the Hardy space of the unit disc will be
denoted by $H^{p}(\mathbb{D})$ and the Hardy space of the upper
half-plane will be denoted by $H^{p}(\mathbb{H})$.

  The two Hardy spaces $H^{2}(\mathbb{D})$ and $H^{2}(\mathbb{H})$
    are isometrically isomorphic. An isometric isomorphism $\Phi:H^{2}(\mathbb{D})\longrightarrow$ $H^{2}(\mathbb{H})$
is given by
   \begin{equation}\label{nice1}
    \Phi(g)(z)=
   \bigg(\frac{1}{\sqrt{\pi}(z+i)}\bigg)g\bigg(\frac{z-i}{z+i}\bigg)
   \end{equation}
The mapping $\Phi$ has an inverse
$\Phi^{-1}:H^{2}(\mathbb{H})\longrightarrow$ $H^{2}(\mathbb{D})$
given by
\begin{equation*}\label{nice2}
\Phi^{-1}(f)(z)= \frac{e^\frac{i\pi}{2}(4\pi)^\frac{1}{2}}{(1-z)}
f\bigg(\frac{i(1+z)}{1-z}\bigg)
\end{equation*}

Using the isometric isomorphism $\Phi$, one may transfer Fatou's
theorem in the unit disc case to upper half-plane and may embed
   $H^{2}(\mathbb{H})$ in $L^{2}(\mathbb{R})$ via
   $f\longrightarrow$ $f^{*}$ where $f^{*}(x)=$
   $\lim_{y\rightarrow 0}f(x+iy)$. This embedding is an
   isometry.

Throughout the paper, using $\Phi$, we will go back and forth
between $H^{2}(\mathbb{D})$ and $H^{2}(\mathbb{H})$. We use the
property that $\Phi$ preserves spectra, compactness and essential
spectra i.e. if $T\in\mathcal{B}(H^{2}(\mathbb{D}))$ then
\begin{equation*}
    \sigma_{\mathcal{B}(H^{2}(\mathbb{D}))}(T)=\sigma_{\mathcal{B}(H^{2}(\mathbb{H}))}(\Phi\circ
    T\circ\Phi^{-1}),
\end{equation*}
$K\in K(H^{2}(\mathbb{D}))$ if and only if $\Phi\circ
K\circ\Phi^{-1} \in K(H^{2}(\mathbb{H}))$ and hence we have
\begin{equation}
    \sigma_{e}(T)=\sigma_{e}(\Phi\circ T\circ\Phi^{-1}).
\end{equation}
We also note that $T\in\mathcal{B}(H^{2}(\mathbb{D}))$ is
essentially normal if and only if $\Phi\circ
T\circ\Phi^{-1}\in\mathcal{B}(H^{2}(\mathbb{H}))$ is essentially
normal.

The Toeplitz operator with symbol $a$ is defined as
    $$T_{a}=P M_{a}|_{H^{2}} ,$$
where $P$ denotes the orthogonal projection of $L^{2}$ onto
$H^{2}$. A good reference about Toeplitz operators on $H^{2}$ is
Douglas' treatise (\cite{Douglas}). Although the Toeplitz
operators treated in \cite{Douglas} act on the
 Hardy space of the unit disc, the results can be transfered
 to the upper half-plane case using the isometric isomorphism $\Phi$
 introduced by equation (1). In the sequel the following identity
 will be used:
\begin{equation}
     \Phi^{-1}\circ T_{a}\circ\Phi=T_{a\circ \mathfrak{C}^{-1}} ,
\end{equation}
where $a\in L^{\infty}(\mathbb{R})$. We also employ the fact
\begin{equation}
    \parallel T_{a}\parallel_{e}=\parallel
    T_{a}\parallel=\parallel a\parallel_{\infty}
\end{equation}
 for any $a\in L^{\infty}(\mathbb{R})$, which is a consequence
 of Theorem 7.11 of \cite{Douglas} (pp. 160--161) and equation (3). For any subalgebra $A\subseteq L^{\infty}(\mathbb{R})$ the Toeplitz C*-algebra generated
 by symbols in $A$ is defined to be
\begin{equation*}
 \mathcal{T}(A)=C^{*}(\{T_{a}:a\in A\}).
\end{equation*}
 It is a well-known result of Sarason (see \cite{Sarason}) that the
 set of functions
\begin{equation*}
 H^{\infty}+C=\{f_{1}+f_{2}:f_{1}\in H^{\infty}(\mathbb{D}),f_{2}\in C(\mathbb{T})\}
\end{equation*}
 is a closed subalgebra of $L^{\infty}(\mathbb{T})$. The following theorem of
 Douglas \cite{Douglas} will be used in the sequel.
\begin{theorem} [\scshape Douglas' Theorem]
    Let $a$,$b\in$ $H^{\infty}+C$ then the semi-commutators
    $$T_{ab}-T_{a}T_{b}\in K(H^{2}(\mathbb{D})),\quad T_{ab}-T_{b}T_{a}\in K(H^{2}(\mathbb{D})), $$
    and hence the commutator
    $$[T_{a},T_{b}]=T_{a}T_{b}-T_{b}T_{a}\in K(H^{2}(\mathbb{D}))$$
    is compact. \label{thmDouglas}
\end{theorem}
Let $QC$ be the C*-algebra of functions in $H^{\infty}+C$ whose
complex conjugates also belong to $H^{\infty}+C$. Let us also
define the upper half-plane version of $QC$ as the following:
\begin{equation*}
    QC(\mathbb{R})=\{\varphi\in L^{\infty}(\mathbb{R}):\varphi\circ\mathfrak{C}^{-1}\in QC\}.
\end{equation*}
Going back and forth with Cayley transform one can deduce that
$QC(\mathbb{R})$ is a closed subalgebra of
$L^{\infty}(\mathbb{R})$.

By Douglas' theorem and equation (3), if $a$, $b\in
QC(\mathbb{R})$, then
\begin{equation*}
T_{a}T_{b}-T_{ab}\in K(H^{2}(\mathbb{H})).
\end{equation*}
Let $scom(QC(\mathbb{R}))$ be the closed ideal in
$\mathcal{T}(QC(\mathbb{R}))$ generated by the semi-commutators
$\{T_{a}T_{b}-T_{ab}:a, b\in QC(\mathbb{R})\}$. Then we have
\begin{equation*}
    com(\mathcal{T}(QC(\mathbb{R})))\subseteq
    scom(QC(\mathbb{R}))\subseteq K(H^{2}(\mathbb{H})) .
\end{equation*}
 By Proposition 7.12 of \cite{Douglas} and equation (3) we have
\begin{equation}
    com(\mathcal{T}(QC(\mathbb{R})))=scom(QC(\mathbb{R}))=K(H^{2}(\mathbb{H})) .
\end{equation}
 Now consider the symbol map
\begin{equation*}
 \Sigma:QC(\mathbb{R})\rightarrow\mathcal{T}(QC(\mathbb{R}))
\end{equation*}
  defined as
 $\Sigma(a)=T_{a}$. This map is linear but not necessarily multiplicative; however if we let $q$ be
 the quotient map
\begin{equation*}
    q:\mathcal{T}(QC(\mathbb{R})) \rightarrow \mathcal{T}(QC(\mathbb{R}))/scom(QC(\mathbb{R})) ,
\end{equation*}
then $q\circ\Sigma$ is multiplicative; moreover by equations (4)
and (5), we conclude that
 $q\circ\Sigma$ is an isometric *-isomorphism from $QC(\mathbb{R})$
 onto $\mathcal{T}(QC(\mathbb{R}))/K(H^{2}(\mathbb{H}))$. The
 maximal ideal space $M(QC(\mathbb{R}))$ is fibered over
 $\dot{\mathbb{R}}$ in the following way:
 For any $x\in$ $M(QC(\mathbb{R}))$ consider $\tilde{x}=$
    $x|_{C(\dot{\mathbb{R}})}$ then $\tilde{x}\in$
    $M(C(\dot{\mathbb{R}}))=$ $\dot{\mathbb{R}}$. Hence
    $M(QC(\mathbb{R}))$ is fibered over $\dot{\mathbb{R}}$, i.e.
\begin{equation*}
    M(QC(\mathbb{R}))=\bigcup_{t\in\dot{\mathbb{R}}}M_{t}(QC),
\end{equation*}
    where
\begin{equation*}
    M_{t}(QC)=\{x\in M(QC(\mathbb{R})):\tilde{x}=x|_{C(\dot{\mathbb{R}})}=\delta_{t},\delta_{t}(f)=f(t)\}.
\end{equation*}

 We also
 remind the reader about the very important fact in the theory of
 Toeplitz operators that any Toeplitz operator $T_{\varphi}$ with
 symbol $\varphi\in L^{\infty}$ is invertible if and only if
 $T_{\varphi}$ is Fredholm and the Fredholm index
 $ind(T_{\varphi})=0$ is zero. The proof of this fact can also be
 found in \cite{Douglas}. This fact will be used in the proof of
 our main result in this paper.

Let $\varphi:\mathbb{D}\longrightarrow$ $\mathbb{D}$ or
$\varphi:\mathbb{H}\longrightarrow$ $\mathbb{H}$ be a holomorphic
self-map of the unit disc or the upper half-plane. The
\emph{composition
 operator} $C_{\varphi}$ on $H^{p}(\mathbb{D})$ or $H^{p}(\mathbb{H})$ with symbol $\varphi$ is defined by
\begin{equation*}
C_{\varphi}(g)(z)= g(\varphi(z)),\qquad
z\in\mathbb{D}\quad\textrm{or}\quad z\in\mathbb{H}.
\end{equation*}
Composition operators of the unit disc are always
 bounded \cite{CoMac} whereas composition operators of the upper half-plane are not
 always bounded. For the boundedness problem of composition operators of the upper half-plane see
 \cite{Matache}.
The composition operator $C_{\varphi}$ on $H^{2}(\mathbb{D})$ is
carried over to
$(\frac{\tilde{\varphi}(z)+i}{z+i})C_{\tilde{\varphi}}$ on
$H^{2}(\mathbb{H})$ through $\Phi$, where $\tilde{\varphi}=$
$\mathfrak{C}\circ\varphi\circ\mathfrak{C}^{-1}$, i.e. we have
\begin{equation}
    \Phi C_{\varphi}\Phi^{-1} =
    T_{(\frac{\tilde{\varphi}(z)+i}{z+i})}C_{\tilde{\varphi}}.
\end{equation}

However this gives us the boundedness of
 $C_{\varphi}:H^{2}(\mathbb{H})$ $\rightarrow$ $H^{2}(\mathbb{H})$ for
\begin{equation*}
    \varphi(z)=pz+\psi(z) ,
\end{equation*}
 where $p > 0$, $\psi\in$ $H^{\infty}$ and $\Im(\psi(z))>\epsilon>0$ for all $z\in\mathbb{H}$:

Let $\tilde{\varphi}:\mathbb{D}\rightarrow$ $\mathbb{D}$ be an
analytic self-map of $\mathbb{D}$ such that $\varphi=$
$\mathfrak{C}^{-1}\circ\tilde{\varphi}\circ\mathfrak{C}$, then we
have
\begin{equation*}
    \Phi C_{\tilde{\varphi}}\Phi^{-1}=T_{\tau}C_{\varphi}
\end{equation*}
where
\begin{equation*}
    \tau(z)=\frac{\varphi(z)+i}{z+i}.
\end{equation*}
If
\begin{equation*}
    \varphi(z)=pz+\psi(z)
\end{equation*}
with $p > 0$, $\psi\in H^{\infty}$ and $\Im(\psi(z)) > \epsilon>
0$, then $T_{\frac{1}{\tau}}$ is a bounded operator. Since $\Phi
C_{\tilde{\varphi}}\Phi^{-1}$  is always bounded we conclude that
$C_{\varphi}$ is bounded on $H^{2}(\mathbb{H})$.

 The Fourier transform $\mathcal{F}f$ of $f\in$
$\mathcal{S}(\mathbb{R})$ (the Schwartz space, for a definition
see \cite{Rudin}) is defined by
\begin{equation*}
(\mathcal{F}f)(t)=\frac{1}{\sqrt{2\pi}}\int_{-\infty}^{+\infty}e^{-itx}f(x)dx.
\end{equation*}
The Fourier transform extends to an invertible isometry from
$L^{2}(\mathbb{R})$ onto itself with inverse
\begin{equation*}
(\mathcal{F}^{-1}f)(t)=\frac{1}{\sqrt{2\pi}}\int_{-\infty}^{+\infty}
e^{itx}f(x)dx.
\end{equation*}
The following is a consequence of a theorem due to Paley and
Wiener \cite{Rudin}. Let $1 < p < \infty$. For $f\in$
$L^{p}(\mathbb{R})$, the following assertions are equivalent:
\begin{enumerate}
    \item[($i$)]  $f\in$ $H^{p}$,
    \item[($ii$)] $\textrm{supp}(\hat{f})\subseteq$ $[0,\infty)$
\end{enumerate}

A reformulation of the Paley-Wiener theorem says that the image of
$H^{2}(\mathbb{H})$ under the Fourier transform is
$L^{2}((0,\infty))$.

 By the Paley-Wiener theorem we observe that the operator
$$D_{\vartheta}=\mathcal{F}^{-1}M_{\vartheta}\mathcal{F}$$
for $\vartheta\in C([0,\infty])$ maps $H^{2}(\mathbb{H})$ into
itself, where $C([0,\infty])$ denotes the set of continuous
functions on $[0,\infty)$ which have limits at infinity. Since
$\mathcal{F}$ is unitary we also observe that
\begin{equation*}
    \parallel D_{\vartheta}\parallel=\parallel M_{\vartheta}\parallel=\parallel\vartheta\parallel_{\infty}
\end{equation*}
Let $F$ be defined as
\begin{equation*}
F =\{D_{\vartheta}\in B(H^{2}(\mathbb{H})):\vartheta\in
C([0,\infty])\} .
\end{equation*}
We observe that $F$ is a commutative C*-algebra with identity and
the map $D:C([0,\infty])\rightarrow F$ given by
\begin{equation*}
D(\vartheta)=D_{\vartheta}
\end{equation*}
is an isometric *-isomorphism by equation above. Hence $F$ is
isometrically *-isomorphic to $C([0,\infty])$. The operator
$D_{\vartheta}$ is usually called a ``Fourier Multiplier.'' We
will also need the fact that, under the Fourier transform the
Beurling type invariant subspace $e^{i\eta x}H^{2}(\mathbb{H})$,
for $\eta>0$, is mapped onto $L_{\eta}^{2}((0,\infty)):=\{f\in
L^{2}(0,\infty):f(t)=0\quad\textrm{for a.e.}\quad t\in (0,\eta)\}$
for all $\eta>0$ and $\mathcal{F}T_{e^{i\eta
x}}\mathcal{F}^{-1}=S_{\eta}$ where
$S_{\eta}:L^{2}((0,\infty))\rightarrow L^{2}((0,\infty))$ is
defined as $S_{\eta}f(t):=f(t-\eta)$ if $t\geq\eta$ and
$S_{\eta}f(t)=0$ if $0\leq t<\eta$. Similarly we have
$\mathcal{F}T_{e^{-i\eta x}}\mathcal{F}^{-1}=S_{\eta}^{\ast}$
where $S_{\eta}^{\ast}f(t)=f(t+\eta)$. Here $T_{e^{i\eta x}}$ and
$T_{e^{-i\eta x}}$ are Toeplitz operators with symbols $e^{i\eta
z}$ and $e^{-i\eta z}$ respectively. Since
$\cup_{\eta>0}L_{\eta}^{2}((0,\infty))$ is dense in
$L^{2}((0,\infty))$, $\cup_{\eta>0}e^{i\eta x}H^{2}(\mathbb{H})$
is also dense in $H^{2}(\mathbb{H})$.

In \cite{Gul1} the first author studied the C*-algebra $\Psi$
generated by Toeplitz operators with $QC$ symbols and Fourier
multipliers with continuous symbols. He proved that the commutator
$[T_{\varphi},D_{\vartheta}]=T_{\varphi}D_{\vartheta}-D_{\vartheta}T_{\varphi}\in
K(H^{2})$ of any Toeplitz operator with $QC$ symbol and a Fourier
multiplier with continuous symbol is compact which implies that
the Calkin algebra $\Psi/K(H^{2})$ is a commutative C*-algebra
with identity. The maximal ideal space $\mathbb{M}$ of
$\Psi/K(H^{2})$ is also studied in \cite{Gul1} and is found to be
$$\mathbb{M}\cong (M_{\infty}(QC(\mathbb{R}))\times[0,\infty])\cup(M(QC(\mathbb{R}))\times\{\infty\})\subset M(QC)\times [0,\infty]$$
where $M_{\infty}(QC):=\{x\in
M(QC):x|_{C(\dot{\mathbb{R}})}=\delta_{\infty},\delta_{\infty}(f)=\lim_{t\rightarrow\infty}f(t)\}$
is the fiber of $M(QC)$ at infinity. The Gelfand transform
$\Gamma$ of $\Psi/K(H^{2})$ looks like
\begin{equation*}
\Gamma\left(\left[\sum
T_{\varphi_{j}}D_{\vartheta_{j}}\right]\right)(x,t)=
\begin{cases}
\sum\hat{\varphi_{j}}(x)\hat{\vartheta_{j}}(t)\quad\textrm{if}\quad
x\in
M_{\infty}(QC(\mathbb{R})) \\
\sum\hat{\varphi_{j}}(x)\hat{\vartheta_{j}}(\infty)\quad\textrm{if}\quad
t=\infty
\end{cases}
\end{equation*}

\section{the main result}

In this section we prove the main result of this paper which
asserts that any sum $T=\sum T_{\varphi_{j}}D_{\vartheta_{j}}+\sum
D_{\nu_{j}}T_{\psi_{j}}\in\Psi$ convergent in the operator norm is
invertible if and only if $T$ is Fredholm and has Fredholm index
zero. This may be regarded as a generalization of the fact that
any Toeplitz operator $T_{\varphi}$ with a bounded symbol
$\varphi\in L^{\infty}$ is invertible if and only if $T_{\varphi}$
is Fredholm and has Fredholm index zero. In proving this fact our
main technical tool will be a homotopy $H:[0,1]\rightarrow\Psi$
which carries $T$ to a Toeplitz operator. The homotopy $H$ for
$T=\sum T_{\varphi_{j}}D_{\vartheta_{j}}+\sum
D_{\nu_{j}}T_{\psi_{j}}$ is defined as follows:
$$H(w):=\sum T_{\varphi_{j}}D_{\vartheta_{j}^{w}}+\sum
D_{\nu_{j}^{w}}T_{\psi_{j}}$$ where $w\in (0,1]$,
$\vartheta_{j}^{w}(t)=\vartheta_{j}(t-\ln w)$,
$\nu_{j}^{w}(t)=\nu_{j}(t-\ln w)$ and for $w=0$,
$H(0):=T_{\varphi}$,
$\varphi=\sum\lambda_{j}\varphi_{j}+\sum\mu_{j}\psi_{j}$ and
$\lambda_{j}=\lim_{t\rightarrow\infty}\vartheta_{j}(t)$,
$\mu_{j}=\lim_{t\rightarrow\infty}\nu_{j}(t)$. It is easily seen
that when $T=\sum T_{\varphi_{j}}D_{\vartheta_{j}}+\sum
D_{\nu_{j}}T_{\psi_{j}}$ consists of finite sums, $H$ is
continuous. And thus $H$ keeps on being continuous when $T$
consists of infinite sums which are both convergent in operator
norm. In the proof of this fact we will always work with finite
sums since it is enough to prove it for finite sums. In this
section, unless otherwise stated, $H^{2}$ will always be
understood as $H^{2}(\mathbb{H})$ and $QC$ will always be
understood as $QC(\mathbb{R})$.

Here is our main theorem:

\begin{theorem}
Let $T=\sum T_{\varphi_{j}}D_{\vartheta_{j}}+\sum
D_{\nu_{j}}T_{\psi_{j}}\in\Psi$
be such that $\nu_{j}$, $\vartheta_{j}\in C([0,\infty])$, $\psi_{j}$, $\varphi_{j}\in QC$. Then\\
$T\in\Psi$ is invertible $\Leftrightarrow$ $T$ is Fredholm and
$ind(T)=0$
\end{theorem}
\begin{proof}
($\Rightarrow$): trivial.\\
($\Leftarrow$): Let $T$ be Fredholm and $ind(T)=0$. Let
$H:[0,1]\rightarrow\Psi$ be the homotopy constructed above. Since
$\sigma_{e}(T)=\{\sum\hat{\varphi_{j}}(x)\vartheta_{j}(t)+\sum\hat{\psi_{j}}(x)\nu_{j}(t):x\in
M_{\infty}(QC),t\in
[0,\infty]\}\cup\{\sum\lambda_{j}\hat{\varphi_{j}}(x)+\sum\mu_{j}\hat{\psi_{j}}(x):x\in
M(QC)\}$ and
$\sigma_{e}(H(w))=\{\sum\hat{\varphi_{j}}(x)\vartheta_{j}(t)+\sum\hat{\psi_{j}}(x)\nu_{j}(t):x\in
M_{\infty}(QC),t\in [-\ln
w,\infty]\}\cup\{\sum\lambda_{j}\hat{\varphi_{j}}(x)+\sum\mu_{j}\hat{\psi_{j}}(x):x\in
M(QC)\}$, we have $\sigma_{e}(H(w))\subseteq\sigma_{e}(T)$ for all
$w\in [0,1]$. Hence if $T$ is Fredholm then
$0\not\in\sigma_{e}(T)$ $\Rightarrow$ $0\not\in\sigma_{e}(H(w))$
for all $w\in [0,1]$ which implies that $H(w)$ is Fredholm for all
$w\in [0,1]$. Since $H$ is continuous and $ind$ is continuous on
the set of Fredholm operators, we have
$$ind(H(0))=ind(H(w))=ind(H(1))=ind(T)=0\quad\forall w\in [0,1].$$
Hence $ind(H(0)=ind(T_{\varphi})=0$ which implies that
$T_{\varphi}$ is invertible since any Toeplitz operator with
$L^{\infty}$ symbol is invertible iff it is Fredholm with Fredholm
index zero. Since invertible elements in $\Psi$ form an open
subset and $H$ is continuous, there is a $w_{0}\in (0,1]$ such
that $H(w)$ is invertible for all $0\leq w<w_{0}$. Now suppose
that $H(w_{0})$ is not invertible.

For $\eta=\ln(w_{0})-\ln(w)$, where $w_{0}>w$, we have
$M_{\vartheta_{j}^{w}}=S_{\eta}^{\ast}M_{\vartheta_{j}^{w_{0}}}S_{\eta}$
where $M_{g}:L^{2}((0,\infty))\rightarrow L^{2}((0,\infty))$,
$M_{g}f(t):=g(t)f(t)$ is the multiplication operator. Hence we
have
$$D_{\vartheta_{j}^{w}}=(\mathcal{F}^{-1}S_{\eta}^{\ast}\mathcal{F})D_{\vartheta_{j}^{w_{0}}}(\mathcal{F}^{-1}S_{\eta}\mathcal{F})=T_{e^{-i\eta x}}D_{\vartheta_{j}^{w_{0}}}T_{e^{i\eta x}}.$$
 Since $\psi_{j}$, $\varphi_{j}\in QC$ we have
$T_{\varphi_{j}}T_{e^{-i\eta x}}-T_{e^{-i\eta
x}}T_{\varphi_{j}}\in K(H^{2})$ and $T_{\psi_{j}}T_{e^{i\eta
x}}-T_{e^{i\eta x}}T_{\psi_{j}}\in K(H^{2})$ for all $\eta>0$(see
\cite{Douglas}). Hence we have
\begin{eqnarray*}
& &H(w)=\sum T_{\varphi_{j}}D_{\vartheta_{j}^{w}}+\sum
D_{\nu_{j}^{w}}T_{\psi_{j}}=\\
& &\sum T_{\varphi_{j}}T_{e^{-i\eta
x}}D_{\vartheta_{j}^{w_{0}}}T_{e^{i\eta x}}+\sum T_{e^{-i\eta
x}}D_{\nu_{j}^{w_{0}}}T_{e^{i\eta
x}}T_{\psi_{j}} \\
& &=T_{e^{-i\eta x}}(\sum
T_{\varphi_{j}}D_{\vartheta_{j}^{w_{0}}}+\sum
D_{\nu_{j}^{w_{0}}}T_{\psi_{j}})T_{e^{i\eta
x}}+K(w,w_{0})\\
& &=T_{e^{-i\eta x}}H(w_{0})T_{e^{i\eta x}}+K(w,w_{0}).
\end{eqnarray*}
where $K(w,w_{0})\in K(H^{2})$ is a compact operator.

 Hence we have
 $$H(w)=T_{e^{-i\eta x}}H(w_{0})T_{e^{i\eta x}}+K(w,w_{0})$$
 for some compact operator $K(w,w_{0})\in K(H^{2})$. Since Fredholm index is stable under compact
perturbations this implies that the operator $T_{e^{-i\eta
x}}H(w_{0})T_{e^{i\eta x}}$ is Fredholm with index $0$ for all
$\eta>0$. Hence if $H(w_{0})$ is non-invertible then
$\ker(T_{e^{-i\eta x}}H(w_{0})T_{e^{i\eta x}})\neq\{0\}$ for some
$\eta>0$ since $T_{e^{-i\eta x}}H(w_{0})T_{e^{i\eta x}}$ is
Fredholm with index $0$. If this is not the case i.e. if we have
$\ker(T_{e^{-i\eta x}}H(w_{0})T_{e^{i\eta x}})=\{0\}$ for all
$\eta>0$, since $T_{e^{-i\eta x}}H(w_{0})T_{e^{i\eta x}}$ is
Fredholm with index $0$, $T_{e^{-i\eta x}}H(w_{0})T_{e^{i\eta x}}$
is invertible for all $\eta>0$. And this implies that $H(w_{0})$
maps each $e^{i\eta x}H^{2}$ onto itself in a one to one manner.
Hence $H(w_{0})$ maps $\cup_{\eta>0}e^{i\eta x}H^{2}$ onto itself
in a one to one manner. Since $H(w_{0})$ is Fredholm,
$ran(H(w_{0}))\subseteq H^{2}$ is closed in $H^{2}$ and we have
$\cup_{\eta>0}e^{i\eta x}H^{2}\subseteq ran(H(w_{0}))$. Since
$\cup_{\eta>0}e^{i\eta x}H^{2}$ is dense in $H^{2}$ and
$ind(H(w_{0}))=0$, this implies that $ran(H(w_{0}))=H^{2}$ which
in turn implies that $\ker(H(w_{0}))=\{0\}$ and hence $H(w_{0})$
is invertible which is a contradiction. Hence we should have
$\ker(T_{e^{-i\eta x}}H(w_{0})T_{e^{i\eta x}})\neq\{0\}$ for some
$\eta>0$. This implies that either $\ker(H(w_{0}))\cap(e^{i\eta
x}H^{2})\neq\{0\}$ or $H(w_{0})(e^{i\eta
x}H^{2})\cap\ker(T_{e^{-i\eta x}})\neq\{0\}$ should hold for some
$\eta>0$. Now suppose that $\psi_{j}$, $\varphi_{j}\in H^{\infty}$
for all $j\in\mathbb{N}$. Since  $\psi_{j}$, $\varphi_{j}\in
H^{\infty}$ for all $j$, we have $H(w_{0})(e^{i\eta
x}H^{2})\subseteq(e^{i\eta x}H^{2})$ for all $\eta>0$ and
$\ker(T_{e^{-i\eta x}})=(e^{i\eta x}H^{2})^{\perp}$, we have
$H(w_{0})(e^{i\eta x}H^{2})\cap\ker(T_{e^{-i\eta x}})=\{0\}$ for
all $\eta>0$. So if $H(w_{0})$ is non-invertible we should have
$\ker(H(w_{0}))\cap(e^{i\eta x}H^{2})\neq\{0\}$. Since $H(w_{0})$
is Fredholm, $\ker(H(w_{0}))$ is finite dimensional an
$\eta_{0}:=\max\{\eta>0:\ker(H(w_{0}))\cap(e^{i\eta
x}H^{2})\neq\{0\}\}$ exists. For such $\eta_{0}>0$, let
$H_{0}:=T_{e^{-i\eta_{0}x}}H(w_{0})T_{e^{i\eta_{0}x}}$. Then since
$\ker(H(w_{0}))\cap(e^{i\eta_{0} x}H^{2})\neq\{0\}$, $H_{0}$ is
non-invertible. But since
$\ker(H(w_{0}))\cap(e^{i(\eta_{0}+\delta)x}H^{2})=\{0\}$ for all
$\delta>0$ and $H(w_{0})(e^{i(\eta_{0}+\delta)x}H^{2})\subseteq
e^{i(\eta_{0}+\delta)x}H^{2}$ we have
$\ker(T_{e^{-i(\eta_{0}+\delta)x}}H(w_{0})T_{e^{i(\eta_{0}+\delta)x}})=\ker(T_{e^{-i\delta
x}}H_{0}T_{e^{i\delta x}})=\{0\}$ and this implies that
$T_{e^{-i\delta x}}H_{0}T_{e^{i\delta x}}$ is invertible for all
$\delta>0$. This again implies that $H_{0}$ maps $e^{i\delta
x}H^{2}$ onto itself in a one to one manner which implies that
$H_{0}$ should be invertible. This contradicts our assumption that
$H(w_{0})$ is non-invertible. Hence $H(w_{0})$ should be
invertible. Therefore, in this case(where $\psi_{j}$,
$\varphi_{j}\in H^{\infty}$ for all $j$), if $H(w)$ is invertible
for all $0\leq w<w_{0}$ then $H(w_{0})$ is also invertible. So by
transfinite induction $H(w)$ is invertible for all $w\in [0,1]$
and in particular $H(1)=T$ is invertible.

Now suppose that $\varphi_{j}$ and $\psi_{j}$ are continuous for
all $j$. Since trigonometric polynomials are dense in continuous
functions, it is enough to prove the claim when
$\varphi_{j}(z):=\sum_{k=-m}^{m}a_{k}z^{k}$ and $\psi_{j}(z)=\sum_{k=m}^{m}b_{k}z^{k}$ are
a trigonometric polynomials for all $j$. One can write
$\varphi_{j}$ and $\psi_{j}$ in the form
$$\varphi_{j}(z)=z^{-m}\sum_{k=0}^{m'}a_{j}z^{k}=z^{-m}q_{j}(z),\quad \psi_{j}(z)=z^{-m}p_{j}(z)$$
where $z=(\frac{x-i}{x+i})$ and $q_{j}$ and $p_{j}$ are analytic polynomials.
In this case we have
\begin{eqnarray*}
& &H(w)=T_{e^{-i\eta x}}H(w_{0})T_{e^{i\eta x}}+K(w,w_{0})\\
& &=T_{e^{-i\eta x}}(\sum
T_{\varphi_{j}}D_{\vartheta_{j}^{w_{0}}}+\sum D_{\nu_{j}^{w_{0}}}T_{\psi_{j}})T_{e^{i\eta
x}}+K(w,w_{0})\\
& &=T_{e^{-i\eta x}}(\sum
T_{z^{-m}}T_{q_{j}}D_{\vartheta_{j}^{w_{0}}}+\sum D_{\nu_{j}^{w_{0}}}T_{z^{-m}}T_{p{j}})T_{e^{i\eta
x}}+K(w,w_{0})\\
& &=T_{z^{-m}}T_{e^{-i\eta x}}(\sum
T_{q_{j}}D_{\vartheta_{j}^{w_{0}}}+\sum D_{\nu_{j}^{w_{0}}}T_{p_{j}})T_{e^{i\eta x}}+K_{0}+K(w,w_{0})
\end{eqnarray*}
for some $K_{0}\in K(H^{2})$
since $D_{\nu_{j}}T_{z^{-m}}-T_{z^{-m}}D_{\nu_{j}}\in K(H^{2})\quad\forall j$, $T_{z^{-m}}T_{e^{-i\eta x}}=T_{e^{-i\eta x}}T_{z^{-m}}$ and
the sum is finite. Let $\tilde{H}(w_{0})=\sum T_{q_{j}}D_{\vartheta_{j}^{w_{0}}}+\sum D_{\nu_{j}^{w_{0}}}T_{p_{j}}$. Then we have
 $$H(w)=T_{e^{-i\eta x}}T_{z^{-m}}\tilde{H}(w_{0})T_{e^{i\eta x}}+\tilde{K}(w,w_{0})$$ 
 where $\tilde{K}(w,w_{0})=K_{0}+K(w,w_{0})\in K(H^{2})$. Since $T_{z^{-m}}T_{e^{-i\eta x}}=T_{e^{-i\eta x}}T_{z^{-m}}$ we have $\ker(\tilde{H}(w_{0}))\cap e^{i\eta x}H^{2}\neq\{0\}$ for some $\eta>0$. If this is not the case i.e. if $\ker(\tilde{H}(w_{0}))\cap e^{i\eta x}H^{2}=\{0\}\quad\forall\eta>0$ then $H(w)-\tilde{K}(w,w_{0})=T_{z^{-m}}T_{e^{-i\eta x}}\tilde{H}(w_{0})T_{e^{i\eta x}}$ is invertible for all $\eta>0$. And this implies that $\tilde{H}(w_{0})(e^{i\eta x}H^{2})=z^{m}(e^{i\eta x}H^{2})\quad\forall\eta>0$. Since $ran(\tilde{H}(w_{0})\subseteq H^{2}$ is closed, this implies that $ran(\tilde{H}(w_{0}))=z^{m}H^{2}$ and this in turn implies that $H(w_{0})=T_{z^{-m}}\tilde{H}(w_{0})$ is invertible which contradicts our assumption. Since $\ker(\tilde{H}(w_{0})$ is finite dimensional, an $\eta_{0}:=\max\{\eta>0:\ker(\tilde{H}(w_{0}))\cap(e^{i\eta x}H^{2})\neq\{0\}\}$ exists. Now let $\tilde{H}_{0}:=T_{e^{-i\eta_{0}x}}\tilde{H}(w_{0})T_{e^{i\eta_{0}x}}$, then $\ker(\tilde{H}_{0})\neq\{0\}$. Since $\ker(\tilde{H}(w_{0}))\cap(e^{i(\eta_{0}+\delta)x}H^{2})=\{0\}\quad\forall\delta>0$ we have 
 $$T_{z^{-m}}T_{e^{-i\delta x}}\tilde{H}_{0}T_{e^{i\delta x}}=T_{e^{-i\delta x}}T_{z^{-m}}\tilde{H}_{0}T_{e^{i\delta x}}$$ 
 is invertible for all $\delta>0$. This implies that $\tilde{H}_{0}(e^{i\delta x}H^{2})=z^{m}(e^{i\delta x}H^{2})$ for all $\delta>0$. Since $ran(\tilde{H}_{0})\subseteq H^{2}$ is closed, this implies that $ran(H_{0})=z^{m}H^{2}$. Since $ind(T_{z^{-m}}\tilde{H}_{0})=0$ this implies that $ind(\tilde{H}_{0})=-m$ which in turn implies that $\ker(\tilde{H}_{0})=\{0\}$. This contradiction implies that $H(w_{0})$ is
invertible.

Hence if $\psi_{j}$, $\varphi_{j}\in H^{\infty}$ or $\psi_{j}$, $\varphi_{j}$ are
continuous for all $j$ then $T=\sum
T_{\varphi_{j}}D_{\vartheta_{j}}+\sum D_{\nu_{j}}T_{\psi_{j}}$ is Fredholm with index zero
implies that $T$ is invertible. So if $\psi_{j}$, $\varphi_{j}\in
(H^{\infty}+C)\cap QC=QC$ for all $j$ then $T=\sum
T_{\varphi_{j}}D_{\vartheta_{j}}+\sum D_{\nu_{j}}T_{\psi_{j}}$ is Fredholm with index zero
implies that $T$ is invertible.
\end{proof}
A reinterpretation of this theorem would be relating the
invertibility of a generic element in $\Psi$ to the invertibility
of a related Toeplitz operator which is the corollary below:
\begin{corollary}
For any $T=\sum T_{\varphi_{j}}D_{\vartheta_{j}}+\sum D_{\nu_{j}}T_{\psi_{j}}\in\Psi$ such that
$T$ is Fredholm $\psi_{j}$ $\varphi_{j}\in QC$ and $\nu_{j}$, $\vartheta_{j}\in
C([0,\infty])$ $\forall j\in\mathbb{N}$, $T$ is invertible if and
only if $T_{\varphi}\in\Psi$ is invertible where
$\varphi=\sum\lambda_{j}\varphi_{j}+\sum\mu_{j}\psi_{j}$ and $\mu_{j}=\lim_{t\rightarrow\infty}\nu_{j}(t)$,
$\lambda_{j}=\lim_{t\rightarrow\infty}\vartheta_{j}(t)$.
\end{corollary}
\begin{proof}
If $T$ is invertible then $T$ is Fredholm with $ind(T)=0$. Using
the homotopy $H:[0,1]\rightarrow\Psi$ constructed in the beginning
of this section, since $ind$ is continuous we have $H(w)$ is
Fredholm $\forall w\in [0,1]$ and $ind(H(w))=ind(T)=0$ $\forall
w\in [0,1]$. In particular $H(0)=T_{\varphi}$ is Fredholm and
$ind(T_{\varphi})=0$, since any Fredholm Toeplitz operator
$T_{\varphi}$ with $ind(T_{\varphi})=0$ is invertible,
$T_{\varphi}$ is invertible.

On the other hand if $T_{\varphi}$ is invertible, since $T$ is
Fredholm, by the proof of Theorem 3, $H(w)$ is invertible $\forall
w\in [0,1]$, in particular $H(1)=T$ is invertible.
\end{proof}

\section{Applications of the Main Results}

An immediate application of Corollary 4 shows that the essential
spectrum and the spectrum of a quasi-parabolic composition
operator coincide:
\begin{theorem}
Let $\varphi:\mathbb{D}\rightarrow\mathbb{D}$ and
$\tilde{\varphi}:\mathbb{H}\rightarrow\mathbb{H}$ be such that
$\varphi(z)=\frac{2iz+\eta(z)(1-z)}{2i+\eta(z)(1-z)}$ and
$\tilde{\varphi}(w)=w+\psi(w)$ where $\eta\in QC(\mathbb{T})\cap
H^{\infty}$, $\Im(\eta(z))>\delta>0$ for all $z\in\mathbb{D}$,
$\psi\in QC(\mathbb{R})\cap H^{\infty}$, $\Im(\eta(w))>\delta>0$
for all $w\in\mathbb{H}$. Then
$C_{\tilde{\varphi}}:H^{2}(\mathbb{H})\rightarrow
H^{2}(\mathbb{H})$ is bounded and
$\sigma_{e}(C_{\tilde{\varphi}})=\sigma(C_{\tilde{\varphi}})$. We
also have $\sigma_{e}(C_{\varphi})=\sigma(C_{\varphi})$.
\end{theorem}
\begin{proof}
The boundedness of $C_{\tilde{\varphi}}$ was shown in \cite{Gul1}.
In particular we have
$$C_{\tilde{\varphi}}=\sum_{n=0}^{\infty}T_{\tau^{n}}D_{\vartheta_{n}}$$
where $\tau(x)=i\alpha-\psi(x)$ and
$\vartheta_{n}(t)=\frac{(-it)^{n}e^{-\alpha t}}{n!}$ for some
$\alpha>0$. So for any
$\lambda\not\in\sigma_{e}(C_{\tilde{\varphi}})$ we have
$$\lambda-C_{\tilde{\varphi}}=\lambda-\sum_{n=0}^{\infty}T_{\tau^{n}}D_{\vartheta_{n}}$$
where the series on the Right Hand Side converges in the operator
norm. Since
$\lambda_{n}:=\lim_{t\rightarrow\infty}\vartheta_{n}(t)=0$, by
Corollary 1 we have $\lambda-C_{\tilde{\varphi}}$ is invertible if
and only if $\lambda$ is invertible which is certainly the case
since $\lambda\neq 0$. Hence
$\lambda\not\in\sigma(C_{\tilde{\varphi}})$ $\Rightarrow$
$\sigma(C_{\tilde{\varphi}})\subseteq\sigma_{e}(C_{\tilde{\varphi}})$
$\Rightarrow$
$\sigma(C_{\tilde{\varphi}})=\sigma_{e}(C_{\tilde{\varphi}})$. The
same argument applies to $C_{\varphi}$ since
$$\Phi\circ C_{\varphi}\circ\Phi^{-1}=T_{\frac{z+i+\eta\circ\mathfrak{C}(z)}{z+i}}\sum_{n=0}^{\infty}T_{\tilde{\tau}^{n}}D_{\vartheta_{n}}$$
where $\Phi:H^{2}(\mathbb{D})\rightarrow H^{2}(\mathbb{H})$ is the
isometric isomorphism
$$\Phi(f)(z)=\bigg(\frac{1}{\sqrt{\pi}(z+i)}\bigg)f\bigg(\frac{z-i}{z+i}\bigg),$$
$\mathfrak{C}$ is the Cayley transform and
$\tilde{\tau}(x)=i\alpha-\eta\circ\mathfrak{C}(x)$. Hence we also
have $\sigma(C_{\varphi})=\sigma_{e}(C_{\varphi})$.
\end{proof}

\end{document}